\documentclass[oneside,a4paper,11pt]{article}

\usepackage[english]{babel}
\usepackage[latin1]{inputenc}

\usepackage{amsmath,amsthm,amssymb}
\numberwithin{equation}{section}
\usepackage{amsfonts}
\usepackage{multicol}
\usepackage{enumerate}
\usepackage{graphicx}
\usepackage{array}
\usepackage{color}

\usepackage{verbatim} 

\def\l{\lambda}
\def\L{\Lambda}
\def\N{\mathbb N}
\def\Z{\mathbb Z}

\def\R{\mathbb R}

\def\e{\varepsilon}
\def\a{\alpha}

\def\s{\sigma}
\def\l{\lambda}
\def\p{\Phi}
\def\g{\gamma}
\def\im{\mathbf{i}}

\def\CC{\mathcal{C}}

\newtheorem{teor}{Theorem}[section]
\newtheorem{theorem}[teor]{Theorem}

\newtheorem{definition}[teor]{Definition}
\newtheorem{proposition}[teor]{Proposition}
\newtheorem{lemma}[teor]{Lemma}

\newcommand{\norm}[1]{\Vert#1\Vert }

\begin{document}

\title{Compactness of Fourier Integral Operators on weighted modulation spaces}
\author{Carmen Fern\'andez, Antonio Galbis, Eva Primo}
\maketitle

\begin{abstract}
 Using the matrix representation of Fourier integral operators with respect to a Gabor
frame, we study their compactness on weighted modulation spaces. As a consequence, we
recover and improve some compactness results for pseudodifferential operators.
\end{abstract}

\section{Introduction}

The aim of this paper is to investigate compactness for {\it Fourier integral operators} (FIOs) when acting on weighted modulation spaces. The boundedness and Schatten class properties of FIOs have been studied by several authors under various assumptions on the phase and the symbol. See for instance \cite{Bishop,Boulkhemair_1997,Cordero_2009_Boundedness,Cordero_2010_Time,Ruzhansky_2006_Global,Toft_Concetti_2007,Toft_Concetti_2009,Toft_Concetti_2010}. However no characterization seems to be known of those FIOs which are compact. Our approach to the study of the compactness of the FIOs follows the point of view of \cite{Cordero_2010_Time}, which means that our results strongly depend on the matrix representation of a FIO with respect to a Gabor frame.

For a function $f$ on $\R^d$ the FIO $T$ with symbol $\s \in L^{\infty}(\R^{2d})$ and phase $\p$ on $\R^{2d}$ can be formally defined by
$$
Tf(x)= \int_{\R^d}e^{2\pi \im \p(x,\eta)}\s(x,\eta)\hat{f}(\eta)d \eta.
$$

The phase $\p(x,\eta)$ is {\it tame}, which means that it is smooth on ${\mathbb R}^{2d}$ and fulfills the estimates
 \begin{equation}\label{phase-1} |\partial^{\a}_z \p(z)| \leq C_{\a},  \ \ \ |\a|\geq 2, z\in {\mathbb R}^{2d},\end{equation} and the nondegeneracy condition
\begin{equation}\label{phase-2}
 |\det \,  \partial^2_{x,\eta}\, \p(x,\eta)| \geq \delta > 0,\ \ (x,\eta)\in {\mathbb R}^{2d}.
\end{equation}

\par\medskip
The symbol $\sigma$ on ${\mathbb R}^{2d}$ satisfies
\begin{equation}\label{simbolo}
|\partial^{\a}_{z}\s(z)|\leq C_{\a}, \ \ \text{a.e.}\ z\in \R^{2d}, \ |\a|\leq 2N
\end{equation} for a fixed $N\in \N .$
Here $\partial^{\a}_{z}$ denotes the distributional derivative. When $\Phi(x,\eta) = x\eta$ we recover the pseudodifferential operators (PSDOs) in the Kohn-Nirenberg form.
\par\medskip
Let $T_x$ and $M_\omega$ be the translation and modulation operators
$$
\left(T_x f\right)(t) = f(x-t),\ \ \left(M_\omega f\right)(t) = e^{2\pi i \omega t}f(t).
$$ We fix $\alpha > 0, \beta > 0$ and consider the regular lattice $\Lambda = \alpha{\mathbb Z}^d \times \beta{\mathbb Z}^d$. Then, for $\lambda = \left(\alpha n, \beta m\right)\in \Lambda,$ the time-frequency shift $\pi(\lambda)$ is defined by $\pi(\lambda) f = M_{\alpha n}T_{\beta m} f.$ The set of time-frequency shifts ${\mathcal G}(g, \Lambda) = \left\{\pi(\lambda)g:\ \lambda\in \Lambda\right\}$ for a non-zero $g\in L^2({\mathbb R}^d)$ is called a Gabor system.

The key result in \cite{Cordero_2010_Time} shows that the matrix representation of a FIO with respect to a Gabor frame ${\mathcal G}(g, \Lambda)$ with $g\in {\mathcal S}({\mathbb R}^d)$ is well organized. In fact, for a tame phase function $\Phi$ and a symbol $\sigma$ satisfying condition (\ref{simbolo}) there exists a constant $C_N > 0$ such that
\begin{equation}\label{eq:matrix-decay}
\left|\langle T\pi(\l)g, \pi(\mu)g \rangle \right| \leq C_N \langle \chi(\l)-\mu \rangle^{-2N}
\end{equation} for every $\lambda,\mu\in \Lambda$ where $\chi$ is the {\it canonical transformation} of the phase $\Phi.$ We recall that $(x,\xi) = \chi(y,\eta)$ is a bilipschitz map $\chi:\R^d\to\R^d$ defined through the system
$$
\left\{ \begin{matrix}
y= \nabla_\eta \p(x,\eta)\\
\xi = \nabla_x \p(x,\eta)
\end{matrix}.\right.
$$

The estimate (\ref{eq:matrix-decay}) is an extension of previous results of Gr\"ochenig \cite{Grochenig_2006_Time} concerning almost diagonalization of PSDOs. See also \cite{Grochenig_2008_Banach}. The condition (\ref{simbolo}) on the symbol can be relaxed. In fact, if ${\mathcal G}(g, \Lambda)$ is a Parseval frame then the estimate (\ref{eq:matrix-decay}) also holds under the weaker assumption that $\s$ belongs to an appropriate modulation space (see \cite{Cordero_2012_approximation}).

We will use the decay estimate (\ref{eq:matrix-decay}) to discuss the compactness of the FIOs when acting on weighted modulation spaces.  More precisely, we prove that the FIO is compact when acting on some modulation space of the form $M^p_m({\mathbb R}^d)$ if and only if the sequences $$\left(\langle T\pi(\lambda)g, \pi(\chi'(\lambda)+\mu)g\rangle\right)_{\lambda\in\Lambda}$$ converge to zero for all $\mu \in \Lambda,$ where $\chi'$ denotes a discrete version of the canonical transformation $\chi.$ This is the content of Theorem \ref{Teo_FIO_comp}. In particular, it follows that compactness does not depend neither on $p$ nor on $m.$ To achieve our goal we need to focus our attention on a class of matrices $A = \left(a_{\gamma,\gamma'}\right)_{\gamma,\gamma'\in \Lambda}$ with the property that the decay of the coefficient $a_{\gamma,\gamma'}$ is determined by the distance of $(\gamma, \gamma')$ to the graph of $\gamma = \chi(\gamma').$ We characterize when such a matrix defines a compact operator when
acting on weighted $\ell^p$ spaces of sequences. For a quadratic phase $\Phi$ we completely characterize  in Theorem \ref{quadratic} the symbols $\sigma$ satisfying condition (\ref{simbolo}) for which the corresponding FIO is compact. The operators  we are considering may fail to be bounded on mixed modulation spaces as was shown in \cite{Cordero_2010_Time}. To overcome this obstacle, an extra condition on the phase was introduced in \cite{Cordero_2010_Time}. Under this additional condition, the obtained results are extended to weighted mixed modulation spaces. As a consequence, we recover and improve some compactness results for PSDOs obtained in \cite{Fernandez_Galbis_2006,Fernandez_2007_Some,Fernandez_2010_Annihilating}.

\section{Preliminaries}

\subsection{Modulation spaces} The short time Fourier transform (STFT) of a function $f\in L^2({\mathbb R}^d)$ with respect to a non-zero window $g\in L^2({\mathbb R}^d)$ is
$$
V_gf(x, \omega) = \displaystyle \int_{{\mathbb R}^d}f(t)\overline{g(t-x)} e^{-2i \pi \omega t}dt.$$ Clearly, we may also write $V_{g}f(x,\omega)=\left<f,M_{\omega}T_x g\right>, $ where $M_{\omega}$ and $T_{x}$ are the modulation and translation operators. Hence $V_gf$ can also be defined for $f \in {\mathcal S}'({\mathbb R}^d)$ and $g \in {\mathcal S}({\mathbb R}^{d}).$ Modulation space norms are measures of the time-frequency concentration of a function or distribution. In order to quantify the decay properties of the STFT of a distribution $f\in {\mathcal S}'({\mathbb R}^d)$ we will use weight functions. A function $v:\R^N\to (0,\infty)$ is said to be a submultiplicative weight if it is continuous, symmetric on each coordinate and $$v(r+k)\leq v(r)v(k).$$ The polynomial weights are the submultiplicative weights of the form $$
v_s(r)= \langle r\rangle^s = (1+ |r|^2)^{\frac{s}{2}},\ s > 0.$$ A map $m: \R^N \to (0,\infty)$ is said to be $v$-moderate, with constant $C_m$, when $m(r+k)\leq C_m m(r)v(k)$ for every $r, k \in \R^N.$ If $m$ is $v$-moderate, then $1/m$ is also $v$-moderate. Given a non-zero window $g \in {\mathcal S}({\mathbb R}^{d})$, a $v_s$-moderate weight ($s > 0$) $m$ and $1\leq p,q\leq \infty,$ the {\it modulation space} $M^{p,q}_m({\mathbb R}^d)$ consists of all tempered distributions $f\in {\mathcal S}'({\mathbb R}^{d})$ such that
$$
\norm{f}_{M^{p,q}_m} = \norm{V_g f}_{L^{p,q}_m} = \left(\int_{\R^d}\left(\int_{\R^d}\left|V_g(x,\omega)\right|^p m(x,\omega)^p dx\right)^{\frac{q}{p}} d\omega\right)^{\frac{1}{q}}<\infty,
$$ with obvious changes when $p = \infty$ or $q = \infty.$ If $p = q$ we write $M^p_m({\mathbb R}^d)$ instead of $M^{p,p}_m({\mathbb R}^d).$ Then $M^{p,q}_m({\mathbb R}^d)$ is a Banach space whose definition is independent of the window $g.$ For $1\leq p,q < \infty,$ ${\mathcal S}({\mathbb R}^{d})$ is dense in $M^{p,q}_m(\R^d).$

The closure of ${\mathcal S}({\mathbb R}^{d})$ in $M^{\infty,q}_m(\R^d)$ is denoted $M^{0,q}_m(\R^d)$ and $M^{p,0}_m(\R^d)$ is defined similarly. In particular, the closure of ${\mathcal S}({\mathbb R}^{d})$ in $M^\infty(\R^d)$ is denoted by $M^0(\R^d)$ and consists of those tempered distributions whose STFT vanishes at infinity.

For $p, \, q \in [1,\infty)\cup \{0\}$ the dual of $M^{p,q}_m(\R^d)$ can be identified with $M^{p',q'}_{1/m}(\R^d),$ $p'$ and $q'$ being the conjugate exponents of $p$ and $q.$ As usual, we agree that the conjugate exponent of $0$ is $1.$ We refer to \cite{Grochenig_2001_Foundations} for background on modulation spaces.

\subsection{Sequence spaces}
Given  $I$ and $J$  countable sets of indices, a sequence of positive numbers $m = \left(m_{i,j}\right)_{(i,j) \in I\times J}$ and $1 \leq p, q < \infty, $ we consider the sequence space $\ell^{p,q}_m(I\times J)$ consisting of those sequences $x = (x_{i,j})_{(i,j)\in I\times J}$ such that
$$
\|x\|_{\ell^{p,q}_m} := \left(\sum_{j\in J}\left(\sum_{i\in
I}\left|x_{i,j}m_{i,j}\right|^p\right)^{\frac{q}{p}}\right)^{\frac{1}{q}} < \infty.
$$

In the case that $p = \infty$ or $q = \infty,$ the previous norm is modified in the obvious way. If $p = q$ we have the weighted $\ell^p$-spaces.
\par\medskip

We denote by $\ell^{0,q}_m$ the closed subspace of $\ell^{\infty,q}_m$ consisting of those sequences $x\in \ell^{\infty,q}_m$ such that $\lim_{i\in
I}\left|x_{i,j}m_{i,j}\right| = 0$ for every $j\in J.$ $\ell^{p,0}_m$ is defined analogously. It turns out that $\ell^{0,q}_m(I\times J)$ (resp. $\ell^{p,0}_m(I\times J)$) coincides with the closure in $\ell^{\infty,q}_m(I\times J)$ (resp. $\ell^{p,\infty}_m(I\times J)$) of the set of those sequences with finitely many non zero coordinates, denoted ${\mathbb C}^{(I\times J)}.$

Finally, $\ell^{0,0}_m(I\times J)$ coincides with the Banach space $c_{0,m}(I\times J)$ of all sequences $x = (x_{i,j})$ whose product with $m$ converges to $0.$
\par\medskip

 $\ell^{p,q}_m(I\times J)$ ($p, \, q \in [1,\infty]\cup \{0\}$) is a Banach space. For $p, \, q \in [1,\infty)\cup \{0\}$ the dual of $\ell^{p,q}_m(I\times J)$ can be identified with $\ell^{p',q'}_{1/m}(I\times J),$ $p'$ and $q'$ being the conjugate exponents of $p$ and $q.$ As usual, we agree that the conjugate exponent of $0$ is $1.$ The duality is given by
$$
\ell^{p,q}_m(I\times J) \times \ell^{p',q'}_{1/m}(I\times J)\to {\mathbb C},\ (x,y)\mapsto \sum_{i,j}x_{i,j} y_{i,j}.
$$
 \par\medskip

 Given a sequence $a=(a_{i,j})_{(i,j)\in I\times J}$ of complex numbers, we denote by $D_a$ {\it the diagonal operator}
 $$
D_a:{\mathbb  C}^{I\times J}\to {\mathbb  C}^{I\times J}, \,
x=(x_{i,j})_{(i,j)\in I\times J}\mapsto (a_{i,j}x_{i,j})_{(i,j)\in I\times J}.$$ It is well-known that $D_a$ is a bounded operator on $\ell^{p,q}_m(I\times J)$ if  $a\in \ell^\infty(I\times J),$ and moreover $\norm{D_a} = \norm{a}_\infty$ for all $p, \, q \in [1,\infty]\cup \{0\}$ and every $m.$ As a consequence, since each $a\in c_0(I\times J)$ is the $\|\,  \|_\infty$-limit of its finite sections, the diagonal operator $D_a$ is compact on  $\ell^{p,q}_m(I\times J)$ when  $a\in c_0(I\times J).$
\par\medskip

\par\medskip
A lattice on $\R^N$ is a set of the form $\Lambda = A{\mathbb Z}^N,$ where $A$ is an invertible $N\times N$ matrix. Given a submultiplicative weight $v,$  a sequence of positive numbers
 $m=(m_\gamma)_{\gamma \in \Lambda}$ is $v$-moderate, with constant $C_m$, if $m_{\gamma+\gamma'}\leq C_m m_\gamma v(\gamma')$ for all $\gamma,\gamma' \in \Lambda.$ If $m$ is $v$-moderate, then $1/m$ is also $v$-moderate.
For a lattice $\Lambda$  in ${\mathbb R}^{N},$ the translation operator $T_\gamma:{\mathbb C}^{\Lambda}\to {\mathbb C}^{\Lambda}$ is defined by
$$
T_\gamma \left(x_{\lambda}\right)_{\lambda\in \Lambda} = \left(x_{\lambda-\gamma}\right)_{\lambda\in \Lambda}.
$$

\par\medskip

If  $I$ and $J$ are  lattices in ${\mathbb R}^d$  and  ${\mathbb R}^\ell$ respectively,  we write  $\Lambda: = I\times J,$ which is a lattice in ${\mathbb R}^{N}$ ($N = d+\ell$). Given  $m=(m_\gamma)_{i\in \Lambda},$  $v$-moderate with constant $C_m,$  the translation operator  $T_\gamma$ is bounded on $\ell^{p,q}_m(\Lambda)$ for every $\gamma\in \Lambda,$ and $\norm{T_\gamma} \leq C_m v(\gamma).$

\subsection{Gabor frames} We fix a function $g\in L^2({\mathbb R}^d)$ and a lattice $\Lambda = \alpha{\mathbb Z}^d \times \beta{\mathbb Z}^d,$ for $\alpha, \beta > 0.$ The Gabor system ${\mathcal G}(g, \Lambda) = \left\{\pi(\lambda)g:\ \lambda\in \Lambda\right\}$ is said to be a {\it Gabor frame} if there exist constants $A, B > 0$ such that
$$
A\norm{f}_2^2 \leq \sum_{\lambda\in\Lambda}\left|\left<f, \pi(\lambda)g\right>\right|^2 \leq B\norm{f}_2^2\ \ \ \forall f\in L^2({\mathbb R}^d).
$$ If $A = B = 1,$ then the Gabor frame is said to be a {\it Parseval frame}. Associated to the Gabor frame ${\mathcal G}(g, \Lambda)$ we consider the analysis operator
$$
C_g:L^2({\mathbb R}^d)\to \ell^2(\Lambda),\ \ f\mapsto \left(\left<f, \pi(\lambda)g\right>\right)_{\lambda\in \Lambda},
$$ and its adjoint $D_g = C_g^\ast,$ which is the synthesis operator
$$
D_g:\ell^2(\Lambda)\to L^2({\mathbb R}^d),\ \ \left(c_\lambda\right)_{\lambda\in \Lambda}\mapsto \sum_{\lambda\in \Lambda}c_\lambda \pi(\lambda)g.
$$ Then $S_g = D_g\circ C_g$ is a bounded and invertible operator on $L^2({\mathbb R}^d)$ called frame operator. The canonical dual window of $g$ is defined as $h = S_g^{-1}g.$ It turns out that ${\mathcal G}(h, \Lambda)$ is also a Gabor frame and
$$
D_g\circ C_h = D_h\circ C_g = Id_{L^2({\mathbb R}^d)}.
$$ If the Gabor frame is a Parseval frame then $S_g = Id_{L^2({\mathbb R}^d)}$ and $h = g.$
\par\medskip
In the case that ${\mathcal G}(g, \Lambda)$ is a Gabor frame and $g\in {\mathcal S}({\mathbb R}^d)$ then, as proved by Janssen (see \cite{Janssen_1995_duality} or \cite[13.5.4]{Grochenig_2001_Foundations}), also $h = S_g^{-1}(g)\in {\mathcal S}({\mathbb R}^d).$ Gr\"ochenig and Leinert \cite[4.5]{Grochenig_2004_Wiener} showed the existence of Parseval frames ${\mathcal G}(g, \Lambda)$ with $g\in {\mathcal S}({\mathbb R}^d).$ Moreover, for every  polynomially moderate weight $m$ and for every $1\leq p,q\leq \infty,$
$$C_g:M_m^{p,q}({\mathbb R}^d)\to \ell^{p,q}_{m}(\Lambda)\ \mbox{and}\ D_g:\ell^{p,q}_{m}(\Lambda)\to M_m^{p,q}({\mathbb R}^d)$$ are bounded operators, weak$^\ast$ continuous, and $D_g\circ C_h = D_h\circ C_g = Id_{M_m^{p,q}({\mathbb R}^d)}.$ Here $D_g$ is the transposed map of $C_g:M_{1/m}^{p',q'}({\mathbb R}^d)\to \ell^{p',q'}_{1/m}(\Lambda).$ For $p = 1$ or $q = 1$ we take $p' = 0$ or $q' = 0$ respectively.
\par\medskip
If $c = \left(c_\lambda\right)_{\lambda\in \Lambda}$ and $1\leq p,q < \infty$ then $D_g(c) = \sum_{\lambda\in \Lambda}c_\lambda \pi(\lambda)g.$ In the limit cases $p = \infty$ or $q = \infty$ the series in the right hand side converges to $D_g(c)$ in the weak$^\ast$ topology. See for instance \cite{Feichtinger_Grochenig_1997} or \cite[12.2.3,12.2.4]{Grochenig_2001_Foundations}.

\subsection{Matrix representation of operators}\label{sec:matrix_representation} Cordero, Nicola and Rodino \cite{Cordero_2010_Time} obtained a result on almost diagonalization for FIOs with respect to a Gabor frame which permitted to study boundedness of Fourier Integral Operators (FIOs) on weighted modulation spaces. Our aim is to use the almost diagonalization technique to study the compactness of FIOs. To this end we need to establish a clear relationship between operators acting on modulation spaces and operators acting on appropriate sequence spaces.
\par\medskip
From now on we assume that ${\mathcal G}(g, \Lambda)$ is a Gabor frame and $g\in {\mathcal S}({\mathbb R}^d).$ Then $h = S_g^{-1}(g)\in {\mathcal S}({\mathbb R}^d)$ and $D_g\circ C_h = D_h\circ C_g = Id_{M_m^{p,q}({\mathbb R}^d)}$ for all $p,q\in [1,\infty]$ and for every $v$-moderate weight $m.$ The (topological) identities ${\mathcal S}'({\mathbb R}^d) = \bigcup \{M^2_{1/v_s}: s>0\}$ and ${\mathcal S}({\mathbb R}^d) = \bigcap \{M^2_{v_s}: s>0\}$ permit to conclude that
$$
C_g, C_h:{\mathcal S}({\mathbb R}^d)\to s(\Lambda)
$$ and
$$
C_g, C_h:{\mathcal S}'({\mathbb R}^d)\to s'(\Lambda)
$$ are topological isomorphisms into their ranges, where $s(\Lambda)$ is the space  of rapidly decreasing  sequences and $s'(\Lambda),$ its dual space, is endowed with the inductive topology. Moreover, every $f\in {\mathcal S}({\mathbb R}^d)$ admits a decomposition
$$
f = \sum_{\lambda \in \Lambda}\langle f, \pi(\lambda)h\rangle \pi(\lambda)g,$$ where the series converges in ${\mathcal S}({\mathbb R}^d).$
\par\medskip
\begin{definition}{\rm
The {\it Gabor matrix} associated to a continuous and linear operator $T:{\mathcal S}({\mathbb R}^d) \to {\mathcal S}'({\mathbb R}^d)$ is defined as
$$
M(T) = \left(\langle T(\pi(\lambda)g), \pi(\mu)g\rangle\right)_{(\mu, \lambda)\in \Lambda \times \Lambda}.$$ If $T$ is a FIO with symbol $\sigma$ and phase $\Phi$ we write $M(\sigma,\Phi)$ instead of $M(T).$
}
\end{definition}

\begin{theorem}\label{th:operator_versus_matrix}{\rm Let $T:{\mathcal S}({\mathbb R}^d) \to {\mathcal S}'({\mathbb R}^d)$ be a continuous and linear operator and ${\mathcal G}(g, \Lambda)$ a Gabor frame with $g\in {\mathcal S}({\mathbb R}^d).$ Then
\begin{itemize}
 \item[(1)] For $1\leq p,q < \infty,$ $T$ can be (uniquely) extended as a bounded operator from $M^{p,q}_{m_1}({\mathbb R}^d)$ into $M^{p,q}_{m_2}({\mathbb R}^d)$ if and only if $M(T)$ defines a bounded operator from $\ell^{p,q}_{m_1}(\Lambda)$ into $\ell^{p,q}_{m_2}(\Lambda)$.
 \item[(2)] For $1\leq p,q \leq \infty,$ $T$ can be extended as a weak$^\ast$ continuous operator from $M^{p,q}_{m_1}({\mathbb R}^d)$ into $M^{p,q}_{m_2}({\mathbb R}^d)$ if and only if $M(T)$ defines a weak$^\ast$ continuous operator from $\ell^{p,q}_{m_1}(\Lambda)$ into $\ell^{p,q}_{m_2}(\Lambda)$.
 \item[(3)] Let $1\leq p,q\leq \infty$ and assume that $T:M^{p,q}_{m_1}({\mathbb R}^d) \to   M^{p,q}_{m_2}({\mathbb R}^d)$ is weak$^\ast$ continuous. Then $T:M^{p,q}_{m_1}({\mathbb R}^d) \to   M^{p,q}_{m_2}({\mathbb R}^d)$ is compact if and only if $M(T):\ell^{p,q}_{m_1}(\Lambda)\to \ell^{p,q}_{m_2}(\Lambda)$ is.
\end{itemize}
 }
\end{theorem}
\begin{proof}
Let $h$ be the canonical dual window of $g.$ Then we have
$$
C_g\circ T = M(T) \circ C_h\ \ \mbox{on}\ \ {\mathcal S}(\R^d).$$ Clearly, $M(T)$ defines a continuous operator from the range $C_h({\mathcal S}({\mathbb R}^d))$, which is a closed subspace of $s(\Lambda),$ into $s'(\Lambda).$ We now check that $M(T)$ defines a continuous operator from ${\mathbb C}^{(\Lambda)}$ into $s'(\Lambda),$ when ${\mathbb C}^{(\Lambda)}$ is endowed with the topology inherited by $s(\Lambda).$ To this end, we fix $x\in {\mathbb C}^{(\Lambda)}$ and observe that $D_g(x)\in {\mathcal S}({\mathbb R}^d),$ hence $M(T)\circ C_h \circ D_g(x) = C_g\circ T \circ D_g(x).$ That is,
$$
(M(T)(C_h \circ D_g)(x))_\mu = \langle T\left(D_g(x)\right), \pi(\mu)g\rangle = \sum_{\lambda \in \Lambda}\langle T\left(\pi(\lambda)g\right), \pi(\mu)g \rangle\cdot x_\lambda.$$ Consequently, for every finite sequence $x$ we have
$$
M(T)(x) = M(T)(C_h \circ D_g)(x).$$ Therefore, $M(T)$ is continuous on ${\mathbb C}^{(\Lambda)}$ when this space is considered as a subspace of $s(\Lambda).$ By density, $M(T)$ defines a continuous operator from the space $s(\Lambda)$ into $s'(\Lambda),$
$$
M(T):s(\Lambda)\xrightarrow{D_g}S({\mathbb R}^d)\xrightarrow{C_h}C_h({\mathcal S}({\mathbb R}^d))\xrightarrow{M(T)}s'(\Lambda).
$$ \par\medskip
Then we have \begin{equation}\label{op_matrix-1} T = D_h\circ M(T)\circ C_h\ \ \mbox{on}\ \ {\mathcal S}(\R^d)
\end{equation} and
 \begin{equation}\label{op_matrix-2} M(T) = M(T)\circ C_h \circ D_g = C_g\circ T\circ D_g\ \ \mbox{on}\ \ s(\Lambda).
\end{equation}
To prove (1) and (2) we only need to use density or weak$^\ast$ density arguments and the fact that $C_g, C_h:M^{p,q}_m({\mathbb R}^d)\to \ell_m^{p,q}(\Lambda)$ and $D_g, D_h:\ell_m^{p,q}(\Lambda)\to M^{p,q}_m({\mathbb R}^d)$ are bounded for $1\leq p,q < \infty$ and weak$^\ast$ continuous for $1\leq p,q\leq \infty.$
\par\medskip
To finish we prove (3). From the hypothesis we deduce that the identities (\ref{op_matrix-1}) and (\ref{op_matrix-2}) hold on $M^{p,q}_{m_1}(\R^d)$ and $\ell^{p,q}_{m_1}(\Lambda)$ respectively and the conclusion follows.

\end{proof}
\par\medskip
In the applications to the FIOs we will always consider $m_1 = m\circ\chi$ and $m_2 = m.$ In the special case of PSDOs we will have $m_1 = m_2 = m.$

\section{Compactness of FIOs}

\subsection{FIOs on $M^p_m$}\label{sec:Mp}

Our aim is to discuss compactness properties for a FIO $T$ whose phase is tame and with symbol $\sigma\in M^\infty_{1\otimes v_{s_0}}(\R^{2d})$ for some $s_0 > 2d.$ Through this section we fix a lattice $\Lambda = \alpha{\mathbb Z}^d \times \beta{\mathbb Z}^d$ and a Parseval frame ${\mathcal G}(g,\Lambda)$ with $g\in {\mathcal S}({\mathbb R}^d).$ As proved in \cite{Cordero_2012_approximation}, we have an estimate
\begin{equation}\label{eq:matrix-decay-general}
\left|\langle T\pi(\l)g, \pi(\mu)g \rangle \right| \leq C \langle \chi(\l)-\mu \rangle^{-s_0}\ \ \forall \lambda,\mu\in \Lambda.
\end{equation} Observe that any symbol satisfying condition (\ref{simbolo}) belongs to $M^\infty_{1\otimes v_{2N}}.$

 The estimate (\ref{eq:matrix-decay-general}) together with the results of subsection \ref{sec:matrix_representation} suggest that we should consider operators on sequence spaces defined in terms of a matrix $A = \left(a_{\gamma,\gamma'}\right)_{\gamma,\gamma'\in \Lambda}$ with the property that the decay of the coefficient $a_{\gamma,\gamma'}$ is determined by the distance of $(\gamma, \gamma')$ to the graph of $\gamma = \chi(\gamma').$ For convenience we will replace the canonical transformation $\chi$ by an appropriate discrete version $\chi':\Lambda\to \Lambda,$ defined as follows. We fix a symmetric relatively compact fundamental domain $Q$ of $\Lambda$ and, for every $\lambda\in \Lambda,$ decompose any
$\chi(
\lambda) = r_\lambda + \chi'(\lambda)$ where $\chi'(\lambda)\in \Lambda$ and $r_\lambda\in Q.$ Since $\chi^{-1}$ is Lipschitz continuous there is $L > 0$ such that $\chi'(\lambda) = \chi'(\mu)$ implies
$$
a:= 2\sup_{u\in Q}\|u\|\geq \|\chi(\lambda) - \chi(\mu)\|\geq L\|\lambda - \mu\|.
$$ Hence
$$
\chi'^{-1}\left(\{\chi'(\lambda)\}\right) = \left\{\mu\in \Lambda:\ \chi'(\mu) = \chi'(\lambda)\right\}$$ is contained in $\overline{B\left(\lambda,\frac{a}{L}\right)} \cap \Lambda,
$ which is a finite set whose cardinal does not depend on $\lambda$. This suggests the following definition.

\begin{definition}\label{matrix_permutation}{\rm Let $v$ be a submultiplicative weight on ${\mathbb R}^{2d}$ and assume that $\psi:\Lambda \to \Lambda$ satisfies
$$
M = \sup_{\lambda\in \Lambda}\mbox{card}\ \psi^{-1}\left(\{\lambda\}\right) < \infty.$$ We define ${\mathcal C}_{v,\psi}(\Lambda)$ as the set of all matrices $A = \left(a_{\gamma,\gamma'}\right)_{\gamma,\gamma'\in \Lambda}$ such that
$$
\norm{A}_{{\mathcal C}_{v,\psi}} = \sum_{\gamma\in \Lambda}v(\gamma)\cdot\sup_{\lambda\in \Lambda}\left|a_{\psi(\lambda)+\gamma, \lambda}\right| < \infty.$$
 }
\end{definition}

\begin{proposition}\label{prop:Gabor_matrix_in_class}{\rm Let $T$ be a FIO whose phase $\Phi$ is tame and $\sigma\in M^\infty_{1\otimes v_{s_0}}(\R^{2d}),$ $s_0 > 2d.$ Then, for every $0\leq s < s_0-2d$ we have
$$
M(\s,\p)\in \CC_{v_s,\chi'}.
$$
}
\end{proposition}
\begin{proof}
 We put $a_{\mu,\lambda} = \langle T\pi(\l)g, \pi(\mu)g \rangle.$ We have to show that
 $$
 \sum_{\gamma\in \Lambda}v_s(\gamma)\cdot\sup_{\lambda\in \Lambda}\left|a_{\chi'(\lambda)+\gamma,\lambda}\right| < \infty.
 $$ According to \cite[Theorem 3.3]{Cordero_2012_approximation},
\begin{equation*}
\left|\langle T\pi(\l)g, \pi(\mu)g \rangle \right| \leq C \langle \chi(\l)-\mu \rangle^{-s_0}= C (v_{s_0}( \chi(\l)-\mu))^{-1}
\end{equation*} for some constant $C$. Since there is $r_{\l}\in Q$ such that $\chi(\l) = \chi'(\l) + r_{\l}$, we obtain
$$
\begin{array}{*2{>{\displaystyle}ll}}
|a_{\chi'(\lambda)+\gamma, \lambda}| & = |\langle T\pi(\l)g, \pi(\chi'(\l) + \gamma)g \rangle| \\ & \\ &
\leq C (v_{s_0}( \chi(\l)-\chi'(\l)-\gamma))^{-1} \\ & \\ & = \frac{C}{v_{s_0}(r_{\l}-\gamma)} \leq \frac{C v_{s_0}(r_{\l})}{v_{s_0}(\gamma)} \leq \frac{C R}{v_{s_0}(\gamma)}
\end{array}
$$ where $R = \max\{v_{s_0}(r):r \in Q\}.$ Finally, using that $2d < s_0-s,$
$$
\sum_{\gamma\in \Lambda}v_s(\gamma)\cdot\sup_{\lambda\in \Lambda}\left|a_{\chi'(\lambda)+\gamma, \lambda}\right| \leq C R\sum_{\gamma \in \Lambda}\frac{v_s(\gamma)}{v_{s_0}(\gamma)} < \infty.
$$
\end{proof}

The following almost diagonal map will play an important role when discussing compactness properties of operators defined in terms of matrices in $\CC_{v_s,\psi}.$
\begin{definition}\label{permutation}{\rm Let $\psi:\Lambda \to \Lambda$ be as in Definition \ref{matrix_permutation} and $a\in {\mathbb C}^\Lambda.$ Then  $$D_{a,\psi}:{\mathbb C}^\Lambda \to {\mathbb C}^\Lambda$$ is defined by $D_{a,\psi}(x) = y$ where $$
y_\gamma = \left\{\begin{array}{*3{>{\displaystyle}l}}
    0 & \mbox{ if } & \gamma \notin \psi(\Lambda)\\
    &  &  \\
     \sum_{\psi(\lambda)=\gamma}a_\lambda x_\lambda  & \mbox{ if } & \gamma \in \psi(\Lambda)\\
     \end{array}
     \right.$$

In particular, $D_{a,\psi}(e_\gamma)=a_\gamma e_{\psi(\gamma)}.$ Moreover, $D_{a,\psi}\left({\mathbb C}^{(\Lambda)}\right) \subset {\mathbb C}^{(\Lambda)}.$
}
\end{definition}
\par\medskip
We observe that the transposed map
$$D^t_{a,\psi}:{\mathbb C}^{(\Lambda)} \to  {\mathbb C}^{(\Lambda)},$$ is given by
$$
(D^t_{a,\psi}(x))_\lambda = \left(D^t_{a,\psi}(x), e_\lambda\right) = \left(x, a_\lambda e_{\psi(\lambda)}\right) = a_\lambda x_{\psi(\lambda)}.$$ In fact, $D^t_{a,\psi}$ can be extended as a map from ${\mathbb C}^{\Lambda}$ into itself. In the case that $a$ is the constant sequence equal $1$ the map $D_{a,\psi}$ is denoted by $I_{\psi}.$ Then, for an arbitrary $a\in {\mathbb C}^\Lambda$ we have
$$
D_{a,\psi} = I_{\psi}\circ D_a.
$$ When $\psi$ is the identity, $D_{a,\psi}$ is just the diagonal operator $D_a.$

\begin{lemma}\label{lem:descomponerpsi}{\rm Let $$M = \sup_{\gamma\in \Lambda}\mbox{card}\left(\psi^{-1}(\{\gamma\})\right).$$ Then, there is a partition $\Lambda = \bigcup_{j=1}^M \Lambda_j$ with the property that
$\psi$ is injective when restricted to each $\Lambda_j.$
 }
\end{lemma}

Let $m=(m_{\lambda})_{\lambda\in \Lambda}$ be a positive sequence. For any $\psi:\Lambda\to \Lambda$ as in Definition \ref{matrix_permutation} we denote by $m\circ\psi$ the sequence
$$
m\circ\psi = \left(m_{\psi(\lambda)}\right)_{\lambda\in \Lambda}.
$$

\begin{proposition}\label{prop:cont-diagonal_permuted}{\rm Let $\psi :\Lambda \to \Lambda$ be as in Definition \ref{matrix_permutation}, $a=(a_{\lambda})_{\lambda\in \Lambda}$ a sequence of complex numbers, $m=(m_{\lambda})_{\lambda\in \Lambda}$ a positive sequence and $p\in [1,\infty].$ The following conditions are equivalent:
\begin{itemize}
 \item[(1)] $D_{a,\psi}$ is continuous on $\ell^2(\Lambda).$
 \item[(2)] $D_{a,\psi}$ is continuous from $\ell^p_{m\circ\psi}(\Lambda)$ to $\ell^{p}_m(\Lambda).$
 \item[(3)] $a\in \ell^\infty(\Lambda).$
\end{itemize}
}
\end{proposition}
\begin{proof} It suffices to show the equivalence between conditions (2) and (3). Let us assume that condition (2) is satisfied. As $D_{a,\psi}(e_\lambda)=a_\lambda e_{\psi(\lambda)}$ then
$$\|D_{a,\psi}\|\geq \|D_{a,\psi}(\frac{e_\lambda}{m_{\psi(\lambda)}})\|_{\ell^p_m} = |a_\lambda|,$$ from where we get (3).
\par\medskip
To check that (3) implies (2) let us first assume that $a\in \ell^\infty(\Lambda)$ and the restriction of $\psi$ to the support of $a$ is injective. Then
 $$\norm{D_{a,\psi}(x)}_{\ell^p_m} = \norm{ (a_\lambda x_\lambda m_{\psi(\lambda)})_\lambda}_{\ell^p} = \norm{(a_\lambda x_\lambda)_\lambda}_{\ell^p_{m\circ\psi}}\leq \norm{a}_{\ell^\infty} \norm{x}_{\ell^p_{m\circ\psi}}.$$
\par\medskip
In the case that condition (3) is satisfied but $\psi$ is not injective on the support of $a$ we apply Lemma (\ref{lem:descomponerpsi}) and decompose
$$
a = \sum_{j=1}^M a^j,
$$ in such a way that the support of $a^j$ is contained in $\Lambda_j.$ Then $$D_{a,\psi} = \sum_{j=1}^M D_{a^j,\psi}
$$ is continuous from $\ell^p_{m\circ\psi}(\Lambda)$ to $\ell^{p}_m(\Lambda)$ and
$$
\|D_{a,\psi}\|_{\ell^p_ {m\circ\psi}\to\ell^p_m} \leq \sum_{j= 1}^M\|a^j\|_{\ell^\infty}\leq M\|a\|_{\ell^\infty}.
$$ Hence (3) implies (2) is proved.
\end{proof}
\par\medskip
The same argument shows that condition (3) in Proposition \ref{prop:cont-diagonal_permuted} is equivalent to being $D_{a,\psi}$ a bounded operator from $c_{0,m\circ\psi}(\lambda)$ into $c_{0,m}(\lambda).$
\par\medskip
In particular, $I_\psi:\ell^p_{m\circ\psi}(\Lambda) \to \ell^{p}_m(\Lambda)$ is continuous. We observe that, if $p\neq q$, the map $I_\psi$ need not be bounded on spaces $\ell^{p,q}_m(\Lambda).$

\begin{proposition}\label{A_cont_permuted}{\rm Let $m=(m_{\lambda})_{\lambda\in \Lambda}$ a $v$-moderate positive sequence, $A = \left(a_{\gamma,\gamma'}\right)_{\gamma,\gamma'\in \Lambda}\in {\mathcal C}_{v,\psi}(\Lambda)$ and $1\leq p \leq \infty$ be given. Then
\begin{itemize}
 \item[(1)] $A:\ell^{p}_{m\circ \psi}(\Lambda) \to \ell^{p}_{m}(\Lambda)$ is a bounded operator, which is also weak$^\ast$ continuous.
 \item[(2)] $A = \begin{displaystyle}\sum_{\gamma\in \Lambda}( T_\gamma \circ D_{a^\gamma, \psi})\end{displaystyle}$ where $a^\gamma:=(a_{\psi(\lambda)+ \gamma,\lambda})_{\lambda\in \Lambda}.$ The series converges absolutely.
\end{itemize}
}
\end{proposition}
\begin{proof} (1) It is easier to deal with the transposed map, so we first consider $b_{\gamma,\gamma'} = a_{\gamma',\gamma}$ and {\it claim} that $B = \left(b_{\gamma,\gamma'}\right)_{\gamma,\gamma'\in \Lambda}$ defines a bounded operator $B:\ell^{q}_{\frac{1}{m}}(\Lambda) \to \ell^{q}_{\frac{1}{m}\circ \psi}(\Lambda)$ for every $1\leq q\leq \infty.$ We should remark here that the class ${\mathcal C}_{v,\psi}(\Lambda)$ need not be closed under taking transposed. Instead we have
\begin{equation}\label{eq:transposed_matrix}
\sum_{\gamma\in \Lambda}v(\gamma)\cdot\sup_{\lambda\in \Lambda}\left|b_{\lambda,\psi(\lambda)+\gamma}\right| < \infty.
\end{equation}
 Using that for every $\lambda \in \Lambda$ one has $\Lambda= \psi(\lambda)+\Lambda$, we may write
$$
\displaystyle\sum_{\gamma\in\Lambda}\left|b_{\lambda,\gamma}x_\gamma\right| =\sum_{\gamma\in\Lambda}\left|b_{\lambda,\psi(\lambda)+\gamma}x_{\psi(\lambda)+\gamma}\right|.$$ From (\ref{eq:transposed_matrix}) and inequality
\begin{equation}\label{eq:cont-diagonal_permuted}
 1\leq C_m\frac{m_{\psi(\lambda)}}{ m_{\psi(\lambda)+\gamma}}v(\gamma)                                                                                                                                                                                                                                                                                                                                                                                                                                                                                                                    \end{equation} we conclude that $$
B:\ell^{q}_{\frac{1}{m}}(\Lambda) \to {\mathbb C}^\Lambda
$$ is a well-defined operator. To prove that $Bx\in \ell^{q}_{\frac{1}{m}\circ\psi}(\Lambda)$ it is enough to check that
$$
\sum_{\lambda\in \Lambda}\left|\left(B x\right)_\lambda y_{\lambda}\right| < \infty
$$ for every $y\in \ell^{q'}_{m\circ\psi}.$ Here $q'$ is the usual conjugate exponent. To this end we denote $\phi(\gamma) = v(\gamma)\sup_{\lambda}\left|b_{\lambda,\psi(\lambda)+\gamma}\right|.$ Using  again the inequality (\ref{eq:cont-diagonal_permuted}) we obtain
$$
\begin{array}{ll}
\displaystyle\sum_{\lambda\in\Lambda}\left|\left(B x\right)_\lambda y_{\lambda}\right| & \begin{displaystyle}\leq \sum_{\lambda\in\Lambda}\sum_{\gamma\in\Lambda}\left|b_{\lambda,\gamma}x_\gamma y_\lambda\right| = \sum_{\lambda\in\Lambda}\sum_{\gamma\in\Lambda}\left|b_{\lambda,\psi(\lambda)+\gamma}x_{\psi(\lambda)+\gamma} y_{\lambda\in\Lambda}\right| \end{displaystyle}\\ & \\ &\begin{displaystyle} \leq C_m\sum_{\lambda\in\Lambda}\sum_{\gamma\in\Lambda}\left|b_{\lambda,\psi(\lambda)+\gamma}\right| v(\gamma)\frac{\left|x_{\psi(\lambda)+\gamma}\right|}{m_{\psi(\lambda)+\gamma}}\left|y_\lambda\right|m_{\psi(\lambda)}\end{displaystyle}\\ & \\ & \begin{displaystyle}\leq C_m \sum_{\gamma\in\Lambda} \phi(\gamma)\sum_{\lambda\in\Lambda}\frac{\left|x_{\psi(\lambda)+\gamma}\right|}{m_{\psi(\lambda)+\gamma}}\left|y_\lambda\right|m_{\psi(\lambda)}\end{displaystyle}\\ & \\ & \begin{displaystyle}\leq M C_m\|x\|_{\ell^{q}_{\frac{1}{m}}}\cdot\|y\|_{\ell^{q'}_{m\circ\psi}}\cdot \|A\|_{{\mathcal C}_{v,\psi}},\end{displaystyle}
\end{array}
$$ where $M$ is the constant in Definition \ref{matrix_permutation}. The claim is proved. Moreover, $B$ also defines a bounded operator from $c_{0,\frac{1}{m}}$ to $c_{0,\frac{1}{m}\circ\psi}.$ In fact, $B:\ell^\infty_{\frac{1}{m}}\to\ell^\infty_{\frac{1}{m}\circ \psi}$ is continuous, $B\left({\mathbb C}^{(\Lambda)}\right)\subset \ell^1_{\frac{1}{m}\circ \psi}\subset c_{0,\frac{1}{m}\circ\psi}$ and $c_{0,\frac{1}{m}}$ is the closure of ${\mathbb C}^{(\Lambda)}$ on $\ell^\infty_{\frac{1}{m}}.$ Consequently, for every $1\leq p\leq \infty,$ the transposed map defines a bounded operator $ A = B^t:\ell^{p}_{m\circ \psi}(\Lambda) \to \ell^{p}_{m}(\Lambda)$ which is also weak$^\ast$ continuous. The proof of (1) is complete.
\par\medskip
(2) Since $m$ is $v$-moderate with constant $C_m$ we have
$$
\|T_\gamma:\ell^{p}_{m}(\Lambda)\to \ell^{p}_{m}(\Lambda)\|\leq C_m v(\gamma).$$ Also
$$
\|D_{a^\gamma , \psi}:\ell^{p}_{m\circ\psi}(\Lambda)\to \ell^{p}_{m}(\Lambda)\| \leq M \sup_\lambda\left|a_{\psi(\lambda)+\gamma,\lambda}\right|, $$ where $M$ is the constant in Definition \ref{matrix_permutation}. Hence
$$
\begin{displaystyle}\sum_{\gamma \in \Lambda}\|T_\gamma \circ D_{a^\gamma, \psi}\|\leq M\sum_{\gamma \in \Lambda}C_m v(\gamma)\sup_{\lambda\in \Lambda}|a_{\psi(\gamma)+\lambda,\lambda}|\end{displaystyle} < \infty.$$ Consequently
$$
S:= \sum_{\gamma\in \Lambda}( T_\gamma \circ D_{a^\gamma,\psi})
$$ defines a bounded operator from $\ell^{p}_{m\circ\psi}(\Lambda)$ into $\ell^{p}_{m}(\Lambda).$ With a similar argument we can decompose the transposed map in terms of operators $D^t_{a^\gamma,\psi} = I_\psi^t\circ D_{a^\gamma},$ from where we conclude that $S$ is also weak$^\ast$ continuous. Moreover, $A$ and $S$ coincide on $\{e_\lambda: \lambda \in \Lambda \},$ from where the result follows. In fact,
$$
\begin{array}{*2{>{\displaystyle}l}}
 \left<S(e_\lambda), e_\mu\right> & = \left<\sum_{\gamma\in\Lambda} a_{\psi(\lambda)+\gamma,\lambda} e_{\psi(\lambda)+\gamma}, e_\mu\right> = \left<\sum_{t\in\Lambda} a_{t,\lambda} e_t, e_\mu\right> \\ & \\ & = \left<A(e_\lambda), e_\mu\right>.
\end{array}
$$

\end{proof}

The following abstract result will be useful to obtain necessary conditions for the compactness of FIOs. We include a proof for the convenience of the reader.

\begin{proposition}\label{convpredual}{\rm Let $E$ and $F$ be Banach spaces and $T:E \to F$ a compact operator. We assume that $E=G'$ and $F= R'$ are dual Banach spaces, $T^{t}(R)\subseteq G$ and $\{x_i\}_{i \in I}$ is a sequence that converges to $x$ in the weak$^\ast$ topology $\sigma (E, G).$ Then $\{(T(x_i))\}_{i \in I}$ converges to $T(x).$
}
\end{proposition}
\begin{proof} We first check that $\{x_{i}\}_{i \in I}$ is a bounded sequence in $E.$ In fact,
$\{x_{i}\}_{{i} \in I}$ is a bounded sequence in $\sigma (E, G)$. If we consider the sequence of linear operators $\{\langle x_{i}, \cdot
\rangle\}_{{i} \in I}$, then for every $g \in G$, $\{\langle x_{i}, g
\rangle\}_{{i} \in I}$ is a bounded sequence. By Banach-Steinhaus' Theorem,
we obtain that $\{\langle x_{i}, \cdot \rangle\}_{{i} \in I}$ is uniformly bounded and we conclude that $\{x_{i}\}_{{i} \in I}$ is a bounded sequence in $E$. We assume that $\{T(x_i)\}_{i \in I}$ does not converge to $T(x)$ in norm. Then there are $\e>0$ and a sequence of indices $(i_k)_{k=1}^\infty \subset I$
such that, for every $k$,
$$
\|T(x_{i_k})-T(x)\|>\e.
$$
Since $T$ is a compact operator, there exists a subsequence $\{T(x_{{i_k}_t})\}_{t}$ converging to some $y \in F.$
Since $\{x_{{i_k}_t}\}_{t}$ $\sigma(E,G)$-converges to $x$ we conclude that
$\{T(x_{{i_k}_t})\}_{t}$ $\sigma(F,R)$-converges to $T(x).$ Since the norm convergence implies the $\sigma(F,R)$-convergence in $F$,we finally obtain that $y = T(x)$. Consequently, $\{T(x_{{i_k}_t})\}_{t}$ converges to $T(x)$ in norm, which is a contradiction.
\end{proof}

\begin{theorem}\label{compac_A_permutation}{\rm
Let $A = \left(a_{\gamma,\gamma'}\right)_{\gamma,\gamma'\in \Lambda}\in {\mathcal C}_{v,\psi}(\Lambda)$ and  $1\leq p \leq \infty$ be given. Then, $A:\ell^{p}_{m\circ\psi}(\Lambda) \to \ell^{p}_{m}(\Lambda)$ is a compact operator if and only if
 $$
a^\gamma:= \left(a_{\psi(\lambda)+\gamma,\lambda}\right)_{\lambda\in \Lambda}\in c_0(\Lambda)\ \ \forall \gamma\in \Lambda.$$
}
\end{theorem}
\begin{proof}
If $a^\gamma \in c_0(\Lambda)$ for every $\gamma \in \Lambda,$  then $D_{a^\gamma, \psi}= I_{\psi}\circ D_{a^\gamma}$ is compact for each $\gamma \in \Lambda.$ Hence, we can apply Proposition \ref{A_cont_permuted} to conclude that $A$ is a compact operator.
\par\medskip
Let us now assume that $A$ is compact. Since $\left(\frac{e_\lambda}{m_{\psi(\lambda)}}\right)_{\lambda\in \Lambda}$ converges to zero in the weak$^\ast$ topology of $\ell^{p}_{m\circ\psi}(\Lambda),$ we can apply Proposition \ref{convpredual} to conclude that $\left(A(\frac{e_\lambda}{m_{\psi(\lambda)}}\right)_{\lambda \in \Lambda}$ converges to $0$. Now, we fix $\gamma \in \Lambda$ and use that $$\frac{m_{\psi(\lambda)+\gamma}}{m_{\psi(\lambda)}}|a_{\psi(\lambda)+\gamma,\lambda}|\leq \norm{A(\frac{e_\lambda}{m_{\psi(\lambda)}})}_{\ell^{p}_{m}(\Lambda)}.$$ Since $m$ is $v$-moderate we obtain $$\left|a_{\psi(\lambda)+\gamma,\lambda}\right|\leq C_mv(\gamma)\norm{A(\frac{e_\lambda}{m_{\psi(\lambda)}})}_{\ell^{p}_{m}(\Lambda)},$$ which finishes the proof.
\end{proof}

We will apply Theorem \ref{compac_A_permutation} to the study of compactness of FIOs
$$
Tf(x)= \int_{\R^d}e^{2\pi \im \p(x,\eta)}\s(x,\eta)\hat{f}(\eta)d \eta
$$ whose phase is tame and with symbol $\sigma\in M^\infty_{1\otimes v_{s_0}}(\R^{2d}),$ $s_0> 2d.$ As usual, $\chi$ is the canonical transformation of the symbol $\Phi$ and $\chi':\Lambda\to \Lambda$ is its discrete version.

\begin{theorem}\label{Teo_FIO_comp_L2}{\rm Let $T$ be a FIO whose phase $\Phi$ is tame and $\sigma\in M^\infty_{1\otimes v_{s_0}}(\R^{2d}),$ $s_0 > 2d.$ The following conditions are equivalent:
\begin{itemize}
\item[(1)] $T:L^2(\R^d) \to L^2(\R^d)$ is a compact operator.
\item[(2)] $M(\s,\p): \ell^2(\Lambda) \to \ell^2(\Lambda)$  is compact.
\item[(3)] $\left(\langle T\pi(\lambda)g, \pi(\chi'(\lambda)+\mu)g\rangle\right)_\lambda\in c_0(\Lambda)$ for every $\mu \in \Lambda$.
\end{itemize}
}
\end{theorem}
\begin{proof}
 Since $M^\infty_{1\otimes v_{s_0}}(\R^{2d})\subset M^{\infty,1}(\R^{2d})$ we can apply \cite[Theorem 6.1]{Cordero_2010_Time} to obtain that $T:L^2(\R^d) \to L^2(\R^d)$ is a bounded operator.  From Theorem \ref{th:operator_versus_matrix} we get the equivalence of conditions (1) and (2). Now it suffices to apply Proposition \ref{prop:Gabor_matrix_in_class} and Theorem \ref{compac_A_permutation} to conclude.
\end{proof}

We observe that, for any positive and $v_s$-moderate weight $m,$
$$
\ell^p_{m \circ \chi}(\Lambda)= \ell^p_{m \circ \chi'}(\Lambda)
$$ with equivalent norms and that $m\circ\chi$ is $v_s$-moderate whenever $m$ is.

\begin{theorem}\label{Teo_FIO_comp}{\rm Let $T$ be a FIO whose phase $\Phi$ is tame and $\sigma\in M^\infty_{1\otimes v_{s_0}}(\R^{2d}),$ $s_0 > 2d.$ Then, for every $0\leq s < s_0-2d$, the following conditions are equivalent:
 \begin{itemize}
\item[(1)] $T:L^2(\R^d) \to L^2(\R^d)$ is a compact operator.
\item[(2)] $T:M^p_{m \circ \chi}(\R^d) \to M^p_m(\R^d)$ is a compact operator for some $1 \leq p < \infty$ and for some $v_s$-moderate weight $m.$
\item[(3)] $T:M^p_{m \circ \chi}(\R^d) \to M^p_m(\R^d)$ is a compact operator for every $1 \leq p < \infty$ and for every $v_s$-moderate weight $m.$
\end{itemize}
}
\end{theorem}
\begin{proof} From \cite[Corollary 5.5]{Cordero_2012_approximation} and Propositions \ref{prop:Gabor_matrix_in_class} and \ref{A_cont_permuted} we have that
$$
T:M^p_{m \circ \chi}(\R^d) \to M^p_m(\R^d)\ \mbox{and}\  M(\sigma,\Phi):\ell^p_{m\circ\chi'}(\Lambda)\to \ell^p_m(\Lambda)
$$ are bounded operators for every $1 \leq p < \infty$ and for every $v_s$-moderate weight $m.$ It suffices to show $(2)\Rightarrow (3).$ According to Theorems \ref{th:operator_versus_matrix} and \ref{compac_A_permutation}, condition (2) is equivalent to the fact that
$$\left(\langle T\pi(\lambda)g, \pi(\chi'(\lambda)+\mu)g\rangle\right)_\lambda\in c_0(\Lambda)$$ for every $\mu \in \Lambda$ and this condition does not depend on $p$ nor on $m.$
\end{proof}

\par\medskip
We next discuss the case $p = \infty.$

\begin{theorem}{\rm Let $T$ be a FIO whose phase $\Phi$ is tame and $\sigma\in M^\infty_{1\otimes v_{s_0}}(\R^{2d})$ and let $0\leq s < s_0-2d$ and $m$ a $v_s$-moderate weight. Then
\begin{itemize}
 \item[(1)] $T$ admits a unique extension as a bounded operator $$T:M^\infty_{m \circ \chi}(\R^d) \to M^\infty_m(\R^d)$$ which is also weak$^\ast$-continuous.
\item[(2)] $T:M^\infty_{m \circ \chi}(\R^d) \to M^\infty_m(\R^d)$ is compact if and only if $$\left(\langle T\pi(\lambda)g, \pi(\chi'(\lambda)+\mu)g\rangle\right)_\lambda\in c_0(\Lambda)$$ for every $\mu \in \Lambda.$
\end{itemize}
}
\end{theorem}
\begin{proof}
 (1) In fact, we consider the composition
$$
T:M^\infty_{m \circ \chi}(\R^d) \xrightarrow{C_g}\ell^\infty_{m \circ \chi}(\Lambda)\xrightarrow{M(\s,\Phi)}\ell^\infty_{m}(\Lambda)\xrightarrow{S^\ast}M^\infty_{m}(\R^d),
$$ where $S = C_g:M^1_{1/m}(\R^d)\to \ell^1_{1/m}(\Lambda).$ We observe that all the involved maps are weak$^\ast$-continuous. Since ${\mathcal S}({\mathbb R}^d)$ is weak$^\ast$-dense in $M^\infty_{m \circ \chi}(\R^d)$ the extension is unique.
\par\medskip
(2) By theorem \ref{th:operator_versus_matrix}, $T$ is compact if and only if $M(\s,\Phi):\ell^\infty_{m \circ \chi}\to\ell^\infty_{m}(\Lambda)$ is. Now it suffices to apply Theorem \ref{compac_A_permutation}.
\end{proof}

For the proof of the next result we recall that the canonical transformation $(x,\xi) = \chi(y,\eta)$ is defined through the system
$$
\left\{ \begin{matrix}
y= \nabla_\eta \p(x,\eta)\\
\xi = \nabla_x \p(x,\eta)
\end{matrix}.\right.
$$

\begin{theorem}\label{comp_si_s0}{\rm Let $T$ be a FIO whose phase $\Phi$ is tame and $\sigma\in M^\infty_{1\otimes v_{s_0}}(\R^{2d})$ and let $0\leq s < s_0-2d.$ If $\s\in M^0(\R^{2d})$ then $T:M^p_{m \circ \chi}(\R^{d}) \to M^p_{m}(\R^{d})$ is a compact operator for every $1\leq p \leq\infty$ and for every $v_s$-moderate weight $m.$
 }
\end{theorem}
\begin{proof} It suffices to show that $M(\s,\Phi)_{\mu,\lambda}$ goes to zero as $|(\lambda,\mu)|$ goes to infinity. To this end we first recall the relation between the Gabor matrix of $T$ and the STFT of $\sigma.$ We denote $\l=(\l_1,\l_2), \mu=(\mu_1,\mu_2) \in \R^{2d}.$ From \cite[(39)]{Cordero_2010_Time} we have
\begin{equation}\label{iguMV}
| \langle T \pi(\l)g, \pi(\mu)g \rangle| =  |V_{\Psi_{\mu_1,\l_2}}\s(z_{\l,\mu})|
\end{equation}
where $$z_{(\l_1,\l_2),(\mu_1,\mu_2)}= ( \mu_1, \l_2, (\mu_2- \nabla_x \p(\mu_1, \l_2)),(\l_1- \nabla_{\eta} \p(\mu_1, \l_2))),$$
$$\Psi_{(\mu_1,\l_2)}(w)= e^{2 \pi \im \p_{2,(\mu_1,\l_2)}(w)}\overline{g} \otimes \hat{g}$$ and $\p_{2,(\mu_1,\l_2)}$ denotes the reminder of order two of the Taylor series of $\p$, that is,
\begin{equation*}
\p_{2,(\mu_1,\l_2)}(w)= 2 \sum_{|\a|=2}\int_0^1 (1-t)\partial^{\a}\p((\mu_1,\l_2)+tw)dt \frac{w^{\a}}{\a!}
\end{equation*}
with $(\mu_1,\l_2), w \in \R^{2d}$. By \cite[6.1]{Cordero_2010_Time} we obtain that
$$D = \{ \Psi_{(\mu_1,\l_2)}: (\mu_1,\l_2)\in \Z^{2d}\}$$ is a relatively compact set in $S(\R^{2d})$. Since $\s \in M^0(\R^{2d}),$
$$
S(\R^{2d}) \to C_0(\R^{2d}),\ \  \Psi \mapsto V_{\Psi}\s
$$ is a continuous map, hence
$$\widetilde{D}= \{ V_{\Psi_{(\mu_1,\l_2)}}\s: (\mu_1,\l_2)\in \Z^{2d}\} = \{V_{\Psi}\s : \Psi \in D\}$$ is a relatively compact set in $C_0(\R^{2d}).$ We conclude that $|V_{\Psi_{\mu_1,\l_2}}\s(z_{\l,\mu})|$ goes uniformly to zero as $|z_{\l,\mu}|$ goes to infinity.\\

Finally, we check that $M(\s,\p)_{\l,\mu}$ goes to zero as $|(\l,\mu)|$ goes to infinity. We can distinguish two cases:
\begin{itemize}
\item $\mu_1$ or $\l_2$ goes to infinity. Then also $|z_{\l,\mu}|$ goes to infinity.
\item Neither $\mu_1$ nor $\l_2$ goes to infinity. We can assume that there exist $C>0$ such that $|(\mu_1, \l_2)|\leq C$, from where it follows that $\nabla_{x} \p(\mu_1, \l_2)$ and $\nabla_{\eta} \p(\mu_1, \l_2)$ are bounded. As $|(\l,\mu)|$ goes to infinity then $\mu_2$ or $\l_1$ goes to infinity. From the fact that $\nabla_{x} \p(\mu_1, \l_2)$ and $\nabla_{\eta} \p(\mu_1, \l_2)$ are bounded, we conclude that $|z_{\l,\mu}|$ goes to infinity.
\end{itemize}
From \eqref{iguMV} we deduce that the Gabor matrix $M(\s,\p)_{\l,\mu}$ goes to $0$ as $|(\l,\mu)|$ goes to infinity and the proof is complete.
\end{proof}

\par\medskip
We now prove that the converse is true in the particular case of quadratic phases.
\begin{definition}{\rm
The map $\Phi: \R^{2d} \to \R$ is said to be a quadratic phase if
$$
\Phi(x, \eta)= \frac{1}{2} Ax \cdot x + Bx \cdot \eta + \frac{1}{2} C \eta \cdot \eta + \eta_0 \cdot x - x_0 \cdot \eta
$$ where $x_0, \eta_0 \in \R^d$, $A, B, C$ are symmetric real matrices and $B$ is non degenerate.
}
\end{definition}

\begin{theorem}\label{quadratic}{\rm Let $T$ be a FIO with quadratic phase $\p$ and $\sigma\in M^\infty_{1\otimes v_{s_0}}(\R^{2d})$ and let $0\leq s < s_0-2d.$ Then the following statements are equivalent:
\begin{itemize}
\item[(1)] $\s \in M^0(\R^{2d})$.
\item[(2)] $T:M^p_{m \circ \chi}(\R^{d}) \to M^p_{m}(\R^{d})$ is a compact operator for every $1\leq p \leq \infty$ and for every $v_s$-moderate weight $m.$
\end{itemize}
}
\end{theorem}
\begin{proof} We need to check that $(2)\Rightarrow (1).$ We use the same notation as in the proof of Theorem \ref{comp_si_s0}. Since the phase $\Phi$ is quadratic then all its second partial derivatives are constant. Hence
$$
\p_{2,(0,0)}(w)= \p_{2,(\mu_1,\l_2)}(w)\ \mbox{and}\ \Psi_{(\mu_1,\l_2)}(w)=\Psi_{(0,0)}(w)=\Psi(w)
$$ for every $(\mu_1,\l_2)\in \R^{2d}.$ Consequently
\begin{equation}\label{igualdad00}
| \langle T \pi(\l)g, \pi(\mu)g \rangle| =  |V_{\Psi_{(0,0)}}\s(z_{\l,\mu})|.
\end{equation}
We now proceed in several steps.
\par\medskip
We first prove that
\begin{equation}\label{eq:gaborgoestocero}
 \left(\langle T \pi(\l)g, \pi(\mu)g \rangle\right)_{\l,\mu \in \Lambda}\in c_0(\Lambda\times \Lambda).
\end{equation}
As $M(\s,\p)\in \CC_{v_s, \psi}$ we have
$$
\sum_{\gamma\in\Lambda}v_s(\gamma)\cdot\sup_{\l \in \Lambda}|M(\s,\p)_{\chi'(\l) + \gamma, \l}| < \infty.
$$ In particular
\begin{equation}\label{eq:(1)}
\lim_{|\gamma| \to \infty}\sup_{\l \in \Lambda}|M(\s,\p)_{\chi'(\l) + \gamma, \l}| = 0.
\end{equation}
Since $T$ is a compact operator we can apply Theorems \ref{Teo_FIO_comp_L2} and \ref{Teo_FIO_comp} to get
\begin{equation}\label{eq:(2)}
 \left(M(\s,\p)_{\chi'(\l) + \gamma, \l}\right)_{\l}\in c_0(\Lambda)
\end{equation}
for every $\gamma \in \Lambda.$ Statement (\ref{eq:gaborgoestocero}) is a consequence of conditions (\ref{eq:(1)}) and (\ref{eq:(2)}).
\par\medskip
Secondly, we check that $G(z,w) = \langle T \pi(z)g, \pi(w)g \rangle$ goes to zero as $|(z,w)|$ goes to infinity on ${\mathbb R}^{4d}.$ We have
\begin{equation}\label{rep_frame_T}
\pi(u)g= \sum_{\nu \in \Lambda} \langle \pi(u)g,\pi(\nu)g \rangle \pi(\nu)g.
\end{equation}
As $g \in S(\R^d) \subseteq M^1(\R^d)$, for every relatively compact subset $K \subset \R^{2d}$ there is $B>0$ such that
$$
\sum_{\nu \in \Lambda} \sup_{u\in K}|V_g g(\nu+u)|v_s(\nu)\leq B \|g\|_{M^1(\R^d)}$$ (see \cite[12.2.1]{Grochenig_2001_Foundations}). In particular, we take $K$ a symmetric and relatively compact fundamental domain of $\Lambda$ and define
$$
\alpha(\nu)= \sup_{u\in K}|V_g g(\nu+u)|= \sup_{u\in K}|\langle \pi(-u)g,\pi(\nu)g \rangle|.$$
Then $\alpha\in \ell^1(\Lambda)$.
Given $z,w \in \R^{2d}$ we can decompose $z=\mu + u$ and $w= \lambda + u'$, with $\mu, \lambda \in \Lambda$ and $u, u' \in K.$ From (\ref{rep_frame_T}) we obtain
$$
\begin{array}{*2{>{\displaystyle}l}}
|\langle T\pi(z)g,\pi(w)g \rangle| & = |\langle T\pi(\mu)\pi(u)g,\pi(\lambda)\pi(u')g \rangle| \\ & \\ & \leq \sum_{\nu, \nu'\in \Lambda} |\langle T\pi(\mu+\nu)g,\pi(\lambda+\nu')g \rangle| \alpha(\nu)\alpha(\nu').
\end{array}
$$ Let $\e > 0$ be given, take $A= \sup_{\lambda, \mu \in \Lambda} |\langle T\pi(\lambda)g,\pi(\mu)g \rangle|$, and find $M >0$ such that
$$\sum _{|\nu| > M} \alpha(\nu)< \frac{\e}{3 A\|\a\|_{\ell^1}}.$$
For every $\mu,\lambda \in \Lambda$ we have that
$$
\sum _{|\nu|\geq M, \nu' \in \Lambda} |\langle T\pi(\mu+\nu)g,\pi(\lambda+\nu')g \rangle| \alpha(\nu)\alpha(\nu')$$ is less than or equal to
$$A \sum _{|\nu|\geq M}\alpha(\nu)\sum _{\nu' \in \Lambda}\alpha(\nu') \leq \frac{\e}{3}
$$
and also
$$
\sum _{\nu\in\Lambda, |\nu'|\geq M} |\langle T\pi(\mu+\nu)g,\pi(\lambda+\nu')g \rangle| \alpha(\nu)\alpha(\nu')  \leq \frac{\e}{3}.
$$ Finally, an application of (\ref{eq:gaborgoestocero}) gives
$$
\begin{array}{*1{>{\displaystyle}l}}
|\langle T\pi(z)g,\pi(w)g \rangle|\leq \\  \\
 \leq \frac{2\e}{3} + \sum _{|\nu|, |\nu'|< M} |\langle T\pi(\mu+\nu)g,\pi(\lambda+\nu')g \rangle| \alpha(\nu)\alpha(\nu')\leq\\ \\ \leq \e
\end{array}
$$ for $|z|+|w|$ large enough. The proof that $|\langle T\pi(z)g,\pi(w)g \rangle| \in C_0(\R^{4d})$ is complete.
\par\medskip
We can now finish the proof that $\s\in M^0(\R^{2d}).$ We recall that
\begin{equation}
| \langle T \pi(\l)g, \pi(\mu)g \rangle| =  |V_{\Psi}\s(z_{\l,\mu})|
\end{equation}
for every $\l,\mu \in \R^2$ and consider $V_{\Psi}\s(a,b,c,d)$ with $(a,b,c,d)\in \R^{4d}$. There are unique $e,f \in \R^d$ such that
$$
(a,b,c,d)= (a,b,e-\nabla_x \p(a, b),f-\nabla_{\eta} \p(a, b))=z_{f,b,a,e}.$$ Then $\left|V_\Psi\sigma(a,b,c,d)\right| = \left|G(f,b,a,e)\right|.$ If $|(a,b,c,d)|$ goes to infinity we have two possibilities:
\begin{itemize}
\item $a$ or $b$ goes to infinity. Then $|(f,b,a,e)|$ goes to infinity.
\item Neither $a$ nor $b$ goes to infinity. We can assume that there is $A>0$ such that $|(a, b)|\leq A$, from where it follows that $\nabla_{x} \p(a, b)$ and $\nabla_{\eta} \p(a, b)$ are bounded. As $|(a,b,e-\nabla_x \p(a, b),f-\nabla_{\eta} \p(a, b))|$ goes to infinity and $a, b, \nabla_{x} \p(a, b), \nabla_{\eta} \p(a, b)$ are bounded, we conclude that either $e$ or $f$ goes to infinity. Hence $|(f,b,a,e)|$ goes to infinity.
\end{itemize}
Since $|\langle T\pi(z)g,\pi(w)g \rangle| \in C_0(\R^{4d})$ we can use \eqref{igualdad00} to conclude that $\s \in M^0(\R^{2d})$.
\end{proof}

\subsection{FIOs on $M^{p,q}_m$}

FIOs we are considering may fail to be bounded on mixed modulation spaces as was shown in \cite{Cordero_2010_Time}. The example was a FIO with phase $\p(x,\eta)=(x\eta, \frac{|x|^2}{2}),$ whose canonical transformation is $\chi(y,\eta)=(y, y+\eta).$ It is easy to check that $I_\chi\left(\ell^{2,1}\right)$ is not contained in $\ell^{2,1}.$

To overcome this obstacle, an extra condition on the phase was introduced in \cite{Cordero_2010_Time}, namely
 \begin{equation}\label{phase-3_mix}
\sup_{x',x, \eta} \left|\nabla_x \p(x,\eta)- \nabla_x \p(x',\eta) \right| < \infty.
\end{equation}
If $\chi = (\chi_1, \chi_2)$ is the corresponding canonical transformation, from condition (\ref{phase-3_mix}), $$\chi_2(y,\eta) =  \nabla_x \p(\chi_1(y,\eta),\eta) = \nabla_x \p(0,\eta)+ a(y, \eta),$$
 $a(y, \eta) $ being a bounded function. From now on, ${\mathcal G}(g, \Lambda)$ is a Parseval frame with $g\in {\mathcal S}({\mathbb R}^d),$ $\Lambda_1 = \alpha {\mathbb Z}^d,$ $\Lambda_2 = \beta {\mathbb Z}^d$ and $\Lambda = \Lambda_1 \times \Lambda_2.$ If $Q$ denotes a symmetric relatively compact fundamental domain of the lattice $\Lambda$ then, there are $K \subseteq \L_2$,  finite, and a unique decomposition
$$
\begin{array}{ll}
\displaystyle \chi_1(\l_1,\l_2) &\displaystyle= r_1{(\l_1,\l_2)} + \psi_1{(\l_1,\l_2)},\\
  \\
\displaystyle\chi_2(\l_1,\l_2) &\displaystyle= r_2{(\l_1,\l_2)} + \psi_2{(\l_2)} + a{(\l_1,\l_2)},
\end{array}
$$ for all $(\l_1,\l_2) \in \L ,$ where $(r_1{(\l_1,\l_2)},r_2{(\l_1,\l_2)}) \in Q$, $\psi_1{(\l_1,\l_2)}\in \L_1$, $\psi_2(\l_2)\in \L_2$ and $a{(\l_1,\l_2)} \in K .$ Moreover, from conditions (\ref{phase-1}) and (\ref{phase-2}) it follows that the map $$ \R^d\to \R^d, \, \eta \mapsto  \nabla_x \p(0,\eta),$$ is a  bilipschitz global diffeomorphism, which implies that $$\sup_{\lambda_2\in \Lambda_2}\mbox{card}\ \psi_2^{-1}(\{\lambda_2\}) < \infty.$$ This motivates the following definition.

\begin{definition}\label{admissible}{\rm Let $\psi:\L_1 \times \L_2 \to \L_1 \times \L_2$, $\psi(i,j)= (\psi_1(i,j) ,\psi_2(i,j)),$ be as in Definition \ref{matrix_permutation}. We say that $\psi$ is admissible if there exist a map $\tilde{\psi}_2:\L_2\to \L_2$ as in Definition \ref{matrix_permutation} and a finite set $K\subset \L_2$ such that $$\psi_2(i,j)=\tilde{\psi}_2(j)+a(i,j), \mbox{ for all } (i,j)$$ where $a(i,j)\in K.$
}
\end{definition}

The discrete versi\'on $\chi':\Lambda \to \Lambda$ of the canonical transformation associated to a phase function satisfying conditions (\ref{phase-1}),(\ref{phase-2}), (\ref{phase-3_mix}) is admissible. From now on, given $\psi$ admissible, to simplify the notation, we will write $\psi_2(j)$ instead of $\tilde{\psi}_2(j).$

Given an admissible $\psi: \L_1 \times \L_2 \to \L_1 \times \L_2$,  let   $C$  be the cardinal of the finite set $K$ and
$M>0$ be such that  for each $(i,j) \in \L_1 \times \L_2,$ $\psi^{-1}(\{(i,j)\})$ has at most $M$ elements and  $\psi_2^{-1}(\{j\})$ has at most $M$ elements for every $j\in \L_2.$

\begin{lemma}\label{lem:M1}{\rm Let $\psi$ be admissible, and  $M_1= C \cdot M.$ For each $j \in \L_2$,  we  define  $\psi_{1,j}:\L_1 \to \L_1$,  as $\psi_{1,j}(i):= \psi_1(i,j).$ Then,  for each $i \in \L_1$ the set $\psi_{1,j}^{-1}(\{i\})$ has at most $M_1$ elements.}
\end{lemma}

\begin{proof}
We fix $j\in \L_2$ and $i_0\in \L_1.$ If $\psi_1(i,j) = \psi_1(i_0,j)$ then $\psi(i,j) = \left(\psi_1(i_0,j), \psi_2(j) + a(i,j)\right)$ can take $C$ different values. Hence, there are only $C\cdot M$ possibilities for $i.$
\end{proof}

We start by analyzing the action of the  basic operators $D_{a,\psi}$ on weighted sequence spaces with mixed norm $\ell^{p,q}_m.$ Since $D_{a,\psi}=I_\psi \circ D_a,$ we will study the continuity of $I_\psi$ on these spaces. To this aim, we consider the transposed map $J_{\psi}:= I_{\psi}^t$, with $\psi$ admissible. We observe that for every $\lambda \in \L,$
$$J_{\psi}(x)= (x_{\psi(\lambda)})_{\lambda} .$$

\begin{proposition}\label{prop:cont-j_mix}
{\rm Let $\psi$ be admissible, $m=(m_{i,j})_{(i,j)\in \L}$ a positive sequence and  $p,q\in [1,\infty]\cup \{0\}.$  Then, $J_{\psi}$ is continuous from $\ell^{p,q}_{{m}}(\L)$ to $\ell^{p,q}_{m \circ \psi}(\L).$}
\end{proposition}
\begin{proof}
Let $x \in \ell^{p,q}_{{m}}(\L)$ and put $y = x\cdot m$ and $\gamma=(i,j).$  Then
$$
|y_{\psi(i,j)}| \leq \sum_{k\in\Lambda_2}\left|y_{\psi_1(i,j),\psi_2(j)+k}\right|,
$$ hence
$$
\left(\sum_{i\in\Lambda_1} |y_{\psi(i,j)}|^p\right)^{\frac{1}{p}} \leq \sum_{k\in\Lambda_2}\left(\sum_{i\in\Lambda_1} \left|y_{\psi_1(i,j),\psi_2(j)+k}\right|^p\right)^{\frac{1}{p}} \leq \sum_{k\in\Lambda_2}\left(M_1\sum_{\ell\in\Lambda_1}\left|y_{\ell, \psi_2(j)+k}\right|^p\right)^{\frac{1}{p}}.
$$ Consequently
$$
\begin{array}{*2{>{\displaystyle}ll}}
\left(\sum_{j\in\Lambda_2}\left(\sum_{i\in\Lambda_1} |y_{\psi(i,j)}|^p\right)^{\frac{q}{p}}\right)^{\frac{1}{q}} & \begin{displaystyle}\leq \sum_{k}^{}\left(\sum_{j\in\Lambda_2}^{}\left(M_1\sum_{\ell\in\Lambda_1} \left|y_{\ell, \psi_2(j)+k}\right|^p\right)^{\frac{q}{p}}\right)^{\frac{1}{q}} \end{displaystyle}                                                                                                                                                                                                                                         \\ & \\ & \begin{displaystyle}\leq C\left(M\sum_{h\in\Lambda_2}\left(M_1\sum_{\ell\in\Lambda_1} |y_{\ell,h}|^p\right)^{\frac{q}{p}}\right)^{\frac{1}{q}}\end{displaystyle}\\ & \\ & = C M^{\frac{1}{q}}M_1^{\frac{1}{p}}\|y\|_{p,q},
\end{array}
$$ where we have applied triangular inequality for the norms in $\ell^p$ and $\ell^q$, and the facts that, for each $j\in \L_2$, $\psi_2(j)$ can be repeated at most $M$ times and  $\psi_1(i,j)= \psi_{1,j}(i)$ can be repeated at most $M_1$ times (Lemma \ref{lem:M1}).

As $J_\psi$ maps finite supported sequences into finite supported sequences, the cases  $p=0$ or $q=0$ follow immediately.
\end{proof}
\par\medskip
For $a\in\ell^\infty(\L)$ we obtain, from the decomposition $D_{a,\psi} = J^t_\psi\circ D_a,$ the estimate
$$
\|D_{a,\psi}:\ell^{p,q}_{m \circ \psi}(\L)\to \ell^{p,q}_m(\L)\| \leq  C M^{\frac{1}{q}}M_1^{\frac{1}{p}}\cdot \|a\|_ \infty.
$$

\begin{proposition}\label{A_cont_permuted_mix}{\rm
Let $A = \left(a_{\gamma,\gamma'}\right)_{\gamma,\gamma'\in \Lambda}\in {\mathcal C}_{v,\psi}(\Lambda)$,  with $\psi$ admissible and $1\leq  p,q \leq \infty $ be given. Then,

\begin{itemize}
\item[(1)] $A$ defines a bounded operator$A:\ell^{p,q}_{m \circ \psi}(\Lambda) \to \ell^{p,q}_m(\Lambda),$ which is also weak$^\ast$ continuous.

\item[(2)] $A = \begin{displaystyle}\sum_{\gamma\in \Lambda}( T_\gamma \circ D_{a^\gamma, \psi})\end{displaystyle}$ where $a^\gamma:=(a_{\psi(\lambda)+ \gamma,\lambda})_{\lambda\in \Lambda}.$ The convergence of the series is absolute.
\end{itemize}}
\end{proposition}
\begin{proof}
It is easier to deal with the transposed map, so we first consider $b_{\gamma,\gamma'} = a_{\gamma',\gamma}$ and {\it claim} that $B = \left(b_{\gamma,\gamma'}\right)_{\gamma,\gamma'\in \Lambda}$ defines a bounded operator $B:\ell^{p,q}_{\frac{1}{m}}(\Lambda) \to \ell^{p,q}_{\frac{1}{m}\circ \psi}(\Lambda),$  for all $ p,q \in [1,\infty ] \cup \{0\}.$ We will assume that $p,q\in [1,\infty].$ Then the case that $p = 0$ or $q = 0$ can be obtained as in the proof of Proposition \ref{A_cont_permuted}.

 As $\ell^{p,q}_{\frac{1}{m}}(\Lambda) \subset \ell^\infty_{\frac{1}{m}}(\Lambda), $ by  Proposition \ref{A_cont_permuted},  we obtain that $$ B:\ell^{p,q}_{\frac{1}{m}}(\Lambda) \to {\mathbb C}^\Lambda $$ is a well-defined operator. To prove that $Bx\in \ell^{p,q}_{\frac{1}{m}\circ \psi}(\Lambda)$ it is enough to check that
$$
\sum_{\g\in\Lambda}\left|\left(B x\right)_{\g} y_{\g}\right| < \infty
$$ for every $y\in \ell^{p',q'}_{m \circ \psi}(\L).$ We proceed as in Proposition \ref{A_cont_permuted}. We denote $\phi(\l) = v(\l)\sup_{\g}\left|b_{\g,\psi(\g)+\l}\right|.$ We obtain, using that translations are isometries on the spaces $\ell^{p,q}$,
$$
\begin{array}{*2{>{\displaystyle}ll}}
 \sum_{\g\in\Lambda}\left|\left(B x\right)_{\g} y_{\g}\right| & \begin{displaystyle}\leq C_m\sum_{\l\in\Lambda} \phi(\l)\sum_{\g\in\Lambda} \frac{x_{\psi(\g)+\l}}{m_{\psi(\g)+\l}}|y_\g| m_{\psi(\g)}\end{displaystyle}\\ & \\ & \begin{displaystyle}\leq C_m\sum_{\l\in\Lambda} \phi(\l)\cdot \|J_\psi(x)\|_{\ell_{\frac{1}{m}\circ \psi }^{p,q}}\cdot \|y\|_{\ell_{m\circ \psi}^{p',q'}}.\end{displaystyle}
\end{array}
$$

(2) follows as in Proposition \ref{A_cont_permuted} once continuities and the estimates for the norms of the operators $D_{a^\gamma, \psi}$ are obtained.
\end{proof}

The characterization of compactness obtained in  \ref{compac_A_permutation} extends to mixed spaces when $\psi$ is admissible.

\begin{proposition}\label{compac_A_permutation_mix}{\rm
Let $A = \left(a_{\g,\g'}\right)_{\g,\g'\in \Lambda}\in {\mathcal C}_{v,\psi}(\Lambda)$, $\psi$ admissible and $1\leq p,q \leq \infty$ be given. Then, $A$ defines a compact operator
$$
A:\ell^{p,q}_{m\circ \psi}(\Lambda) \to \ell^{p,q}_{m}(\Lambda)$$
if and only if $a^\gamma:= \left(a_{\psi(\lambda)+\gamma,\lambda}\right)_{\lambda\in \Lambda}\in c_0(\Lambda)\ \ \forall \g\in \Lambda.$
}
\end{proposition}

The next result extends \cite[Theorem 5.2]{Cordero_2010_Time} to weighted modulation spaces and also includes the cases $p = \infty$ or $q = \infty.$

\begin{theorem}\label{Teo_FIO_cont_mix}{\rm Let $T$ be a FIO whose phase $\Phi$ is tame and satisfies condition (\ref{phase-3_mix}), and $\sigma\in M^\infty_{1\otimes v_{s_0}}(\R^{2d})$ with $0\leq s < s_0-2d.$ Then, $T:M^{p,q}_{m \circ \chi}(\R^d) \to M^{p,q}_m(\R^d)$ is a continuous operator for every $1 \leq {p,q} \leq  \infty$ and for every $v_s$-moderate weight $m.$}
\end{theorem}

\begin{theorem}\label{Teo_FIO_comp_mix}{\rm Let $T$ be a FIO whose phase $\Phi$ is tame and satisfies condition (\ref{phase-3_mix}), and $\sigma\in M^\infty_{1\otimes v_{s_0}}(\R^{2d})$ with $0\leq s < s_0-2d.$ The following conditions are equivalent:
 \begin{itemize}
\item[(1)] $T:L^2(\R^d) \to L^2(\R^d)$ is a compact operator.
\item[(2)] $T:M^{p,q}_{m \circ \chi}(\R^d) \to M^{p,q}_m(\R^d)$ is a compact operator for some $1 \leq {p,q} \leq \infty$ and for some $v_s$-moderate weight $m.$
\item[(3)] $T:M^{p,q}_{m \circ \chi}(\R^d) \to M^{p,q}_m(\R^d)$ is a compact operator for every $1 \leq {p,q} \leq \infty$ and for every $v_s$-moderate weight $m.$
\end{itemize}
}
\end{theorem}

\subsection{PSDOs on $M^{p,q}_m$}

Finally, we are going to consider compactness of pseudodifferential operators in Kohn-Nirenberg form. They are a particular case of FIOs when $\Phi(x,y) = x \cdot y,$ and hence $\chi(y,\eta) = (y,\eta).$ If $\Lambda$ is a regular lattice with symmetric relatively compact fundamental domain $Q, $ the map $\chi'$ is the identity, therefore it is  admissible.  The class of matrices $C_{v,\chi'}$ is denoted by  ${\mathcal C}_v = {\mathcal C}_v(\Lambda)$ and consists of all matrices $A = \left(a_{\gamma,\gamma'}\right)_{\gamma,\gamma'\in \Lambda}$ such that
$$
\|A\|_{{\mathcal C}_v} = \sum_{\gamma\in \Lambda}v(\gamma)\cdot\sup_{\lambda\in \Lambda}\left|a_{\lambda,\gamma+\lambda}\right| < \infty.$$ According to \cite[Lemma 3.5]{Grochenig_2006_Time}, ${\mathcal C}_v$ is an algebra. Since the weight $v$ is symmetric, it follows that
$$\sum_{\gamma\in \Lambda}v(\gamma)\cdot\sup_{\lambda\in \Lambda}\left|a_{\lambda,\gamma+\lambda}\right| = \sum_{\gamma\in \Lambda}v(\gamma)\cdot\sup_{\lambda\in \Lambda}\left|a_{\gamma+\lambda,\lambda}\right|.$$ This means that $A\in {\mathcal C}_v$ if and only if $A^{t}\in {\mathcal C}_v.$
Each $A\in {\mathcal C}_v$ defines a bounded operator $$
A:\ell^{p,q}_m(\Lambda) \to \ell^{p,q}_m(\Lambda),$$ for $p,q\in [1,\infty]\cup \{0\}$ and each $v$-moderate sequence $m.$ The compactness of the map is independent on $p,\, q$ and on $m.$  This allows us to improve results obtained in \cite{Fernandez_2007_Some} and \cite{Fernandez_2010_Annihilating}.

For convenience, we state the results for Weyl pseudodifferential operators. We recall that every operator from ${\mathcal S}({\mathbb R}^d)$ into ${\mathcal S}'({\mathbb R}^d)$ can be represented as a pseudodifferential operator $L_\sigma$ with Weyl symbol $\sigma$ and as a pseudodifferential operator in Kohn-Nirenberg form with symbol $\tau.$ We refer to \cite[Chapter 14]{Grochenig_2001_Foundations} where the relation between $\sigma$ and $\tau$ is established. In particular, for $s\geq 0,$  $\sigma \in M^{\infty, 1}_{1\otimes v_s}(\R^{2d})$ if and only if $\tau \in M^{\infty, 1}_{1\otimes v_s}(\R^{2d}).$

\begin{theorem}\label{Comp_no_m}{\rm
Let $\sigma \in M^{\infty, 1}_{1\otimes v_s}(\R^{2d})$  be given. Then the following statements are equivalent:
\begin{itemize}
\item[(1)] $L_{\sigma}:L^2(\R^d) \to L^2(\R^d)$ is compact.
\item[(2)]  $L_{\sigma}:M^{p,q}_m(\R^d) \to M^{p,q}_m(\R^d)$ is compact for all $p,\, q \in [1,\infty]$ and every $v_s$-moderate weight $m.$
\item[(3)] $L_{\sigma}:M^{p,q}_m(\R^d) \to M^{p,q}_m(\R^d)$ is compact for some $p,\, q \in [1,\infty]$ and some $v_s$-moderate weight $m.$
\item[(4)] $\sigma \in M^{0}(\R^{2d})$.
\end{itemize}
}
\end{theorem}
\begin{proof}
 Let ${\mathcal G}(g, \Lambda)$ be a Gabor frame with $g\in {\mathcal S}({\mathbb R}^d)$ and $\Lambda = \alpha {\mathbb Z}^d\times \beta {\mathbb Z}^d$ for $\alpha, \, \beta >0.$ Then, according to \cite[Theorem 3.2]{Grochenig_2006_Time},
$$
M(\sigma):=\left(\langle L_{\sigma}\pi(\lambda)g,\pi(\mu)g \rangle\right)_{(\mu, \lambda) \in \Lambda \times \Lambda}\in {\mathcal C}_{v_s}(\Lambda).$$ Moreover, it follows from (\ref{op_matrix-1}) and (\ref{op_matrix-2}) that $L_{\sigma}:M^{p,q}_m(\R^d) \to M^{p,q}_m(\R^d)$ is compact if and only if $M(\sigma):\ell^{p,q}_m(\Lambda)\to \ell^{p,q}_m(\Lambda)$ is. Now, the equivalences among (1), (2) and (3) follow from Theorem \ref{compac_A_permutation_mix}. Finally, since $M^{\infty, 1}_{1\otimes v_s}(\R^{2d})\subset M^{\infty, 1}(\R^{2d})$ we can apply \cite[Theorem 4.6]{Fernandez_2007_Some} to obtain that condition (1) is equivalent to condition (4).
\end{proof}
\par\medskip
Alternatively we could argue as follows. According to Theorem \ref{compac_A_permutation_mix}, $M(\sigma):\ell^{p,q}_m(\Lambda)\to \ell^{p,q}_m(\Lambda)$ is a compact operator if and only if \begin{equation}\label{diagonals}\left(\langle L_{\sigma}\pi(\lambda)g,\pi(\lambda + \mu)g \rangle\right)_{\lambda \in \Lambda} \in c_0(\Lambda)\end{equation} for every $\mu \in \Lambda.$ By \cite[3.1]{Grochenig_2006_Time},
$$
\left|\langle L_{\sigma}\pi(\lambda)g,\pi(\lambda + \mu) g \rangle\right| = \left|V_{\Phi}\sigma(\lambda +\frac{ \mu}{2},j(\mu))\right|
$$ where $\Phi = W(g,g),$ and $j:{\mathbb R}^{2d} \to {\mathbb R}^{2d} $ is the map $j(\xi, \omega)= (\omega,-\xi).$ This permits to prove that condition (\ref{diagonals}) is equivalent to the fact that $\sigma \in M^{0}(\R^{2d})$.
\par\medskip\noindent
We want to finish with some comments regarding localization operators (see for instance \cite{Fernandez_Galbis_2006,Fernandez_2010_Annihilating} and the references therein). The compact localization operators on $L^2(\R^d)$ were characterized in \cite{Fernandez_Galbis_2006} in terms of the behavior of the STFT of their symbols. The condition there obtained also gives compactness for the localization operators when acting on weighted modulation spaces of Hilbert type $M^2_m(\R^d)$ (\cite[5.6]{Fernandez_2010_Annihilating}). Since every localization operator can be described as a PSDO in Weyl form, Theorem \ref{Comp_no_m} permits to conclude that the compactness of the localization operator on a modulation class $M^p_m(\R^d)$ does not depend on $p$ nor $m.$ This conclusion could no be achieved with the techniques used in \cite{Fernandez_2010_Annihilating}.

\par\medskip\noindent
{\bf Acknowledgement.} The present research was partially supported by the projects MTM2016-76647-P, ACOMP/2015/186 (Spain). The third author wish to thank the Generalitat Valenciana (Project VALi+d Pre Orden 64/2014) for its support.

\noindent
{\sc Author's address}: Departament d'An\`alisi Matem\`atica, Universitat de Val\`encia, Dr. Moliner 50, 46100-Burjassot, Val\`encia (Spain).
\par\medskip\noindent
fernand@uv.es, antonio.galbis@uv.es, eva.primo@uv.es


\begin{thebibliography}{99}

\bibitem{Bishop} S. Bishop; {\em Schatten class Fourier integral operators}, Appl. Comput. Harmon. Anal. {\bf 31} (2011), 205--217.

\bibitem{Boulkhemair_1997} A. Boulkhemair; {\em Remarks on a Wiener type pseudodifferential algebra and Fourier integral operators}, Math. Res. Lett. {\bf 4} (1997), 53--67.

\bibitem{Toft_Concetti_2007} F. Concetti, J. Toft; {\em Trace ideals for Fourier integral operators with non-smooth symbols}, ``Pseudo-differential operators: partial differential equations and time-frequency analysis'', Fields Inst. Commun., {\em Amer. Math. Soc.} {\bf 52} (2007), 255--264.

\bibitem{Toft_Concetti_2009} F. Concetti, J. Toft; {\em Schatten-von Neumann properties for Fourier integral operators with non-smooth symbols I}, Ark. Mat. {\bf 47} (2009), 295--312.

\bibitem{Cordero_2012_approximation} E. Cordero, K. Gr\"ochenig, F. Nicola; {\em Approximation of Fourier integral operators by Gabor multipliers}, J. Fourier Anal. Appl. {\bf 18} (2012), 661--684.

\bibitem{Cordero_2009_Boundedness} E. Cordero, F. Nicola, L. Rodino; {\em Boundedness of Fourier integral operators on ${\mathcal F}L^p$ spaces}, Trans. Amer. Math. Soc. {\bf 361} (2009), 6049--6071.

\bibitem{Cordero_2010_Time} E. Cordero, F. Nicola, L. Rodino; {\em Time-frequency analysis of Fourier integral operators}, Commun. Pure Appl. Anal. {\bf 9} (2010), 1--21.

\bibitem{Feichtinger_Grochenig_1997} H.G. Feichtinger, K. Gr\"ochenig; {\em Gabor frames and time-frequency analysis of distributions}, J. Funct. Anal. {\bf 146} (1997), 464--495.

\bibitem{Fernandez_Galbis_2006} C. Fern\'andez, A. Galbis; {\em Compactness of time-frequency localization operators on $L^2({\mathbb R}^d)$}, J. Funct. Anal. {\bf 233} (2006), 335--350.

\bibitem{Fernandez_2007_Some} C. Fern\'andez, A. Galbis; {\em Some remarks on compact Weyl operators}, Integral Transforms Spec. Funct. {\bf 18} (2007), 599--607.

\bibitem{Fernandez_2010_Annihilating} C. Fern\'andez, A. Galbis; {\em Annihilating sets for the short time Fourier transform}, Adv. Math. {\bf 224} (2010), 1904--1926.

\bibitem{Grochenig_2001_Foundations} K. Gr\"{o}chenig; Foundations of Time-Frequency
Analysis, {\em Birkh\"{a}user} (2001).

\bibitem{Grochenig_2006_Time} K. Gr\"{o}chenig; {\em Time-frequency analysis of {S}j\"ostrand's class}, Rev. Mat. Iberoam. {\bf 22} (2006), 703--724.

\bibitem{Grochenig_2004_Wiener} K. Gr\"{o}chenig, M. Leinert; {\em Wiener's lemma for twisted convolution and Gabor frames}, J. Amer. Math. Soc. {\bf 17} (2004), 1--18.

\bibitem{Grochenig_2008_Banach} K. Gr\"ochenig, Z. Rzeszotnik; {\em Banach algebras of pseudodifferential operators and their almost diagonalization}, Ann. Inst. Fourier {\bf 58} (2008), 2279--2314.

\bibitem{Janssen_1995_duality} A. J. E. M. Janssen; {\em Duality and biorthogonality for Weyl-Heisenberg frames}, J. Fourier Anal. Appl. {\bf 1} (1995), 403--436.

\bibitem{Ruzhansky_2006_Global} M. Ruzhansky, M. Sugimoto; {\em Global L2-boundedness theorems for a class of Fourier integral operators}, Comm. Partial Differential Equations {\bf 31} (2006), 547--569.

\bibitem{Toft_Concetti_2010} J. Toft, F. Concetti, G. Garello; {\em Schatten-von Neumann properties for Fourier integral operators with non-smooth symbols II}, Osaka J. Math. {\bf 47} (2010), 739--786.
\end{thebibliography}
\end{document}